\documentclass[12pt]{amsart}
\usepackage{amssymb,amsthm,amsmath,amsfonts, mathrsfs}
\usepackage{graphicx, xcolor}
\newcommand{\Ltr}{\reflectbox{$\Gamma$}}

\newcommand{\ltr}{{\reflectbox{\hbox{$\scriptscriptstyle\Gamma$}}}}

\newcommand{\rtr}{{\scriptscriptstyle\Gamma}}
\newcommand{\ftr}{{\scriptscriptstyle{\mathrm T}}}
\newcommand{\ab}{\allowbreak}

\makeatletter
\@namedef{subjclassname@2020}{%
  \textup{2020} Mathematics Subject Classification}
\makeatother

\def\tr{\mathrm{tr}\,}
\def\Tr{\mathrm{Tr}\,}
\def\<{\langle}
\def\>{\rangle}
\newcommand{\cP}{\mathcal{P}}

\def\E{{\mathbb E}}

\renewcommand{\top}{{\scriptscriptstyle\mathrm{T}}}
\newcommand{\mb}{\hskip0.083333em}
\newcommand{\Wg}{\mathrm{Wg}}
\newcommand{\sN}{{\scriptscriptstyle N}}
\newcommand{\sR}{{\scriptscriptstyle R}}

\newcommand{\mybullet}{\vcenter{\hbox{$\scriptscriptstyle\bullet$}}}

\newcommand{\be}{\begin{equation}}
\newcommand{\ee}{\end{equation}}
      \newtheorem{theorem}{Theorem}[section]
       
       \newtheorem{corollary}[theorem]{Corollary}
       \newtheorem{lemma}[theorem]{Lemma}

\newtheorem{definition}[theorem]{Definition}

\theoremstyle{remark}
\newtheorem{remark}[theorem]{Remark}

\title[On partial transposes]
{On partial transposes of unitarily\\ invariant random matrices}

\author[Mingo]{James A.~Mingo$^{(*)}$}
\address{Department of Mathematics and Statistics, Jeffery
	Hall, Queen's University, King\-ston, Ontario, K7L 3Z1,
	Canada}

\author[Popa]{Mihai Popa$^{(**)}$}
\address{The University of Texas at San Antonio, Department
	of Mathematics, One UTSA Circle, San Antonio, Texas 78249,
	and \newline ${}\hspace{.35cm}{}$ Institute of Mathematics
	``Simion Stoilow'' of the Romanian Academy, P.O. Box
	1-764, Bucharest, RO-70700, Romania}

\begin{document}

	\begin{abstract}
 We compute the limit distribution of partial transposes (when both the number and the size of blocks tends to infinity) for a large class of ensembles of unitarily invariant random matrices. Furthermore, it is shown the asymptotic freeness relation between the ensembles of random matrices, their transposes and their left and right partial transposes.
	\end{abstract}

\maketitle

\section{Introduction}
{\let\thefootnote\relax%
\footnote{$^{(*)}$Research supported by a Discovery Grant from the
Natural Sciences and Engineering Research Council of
Canada, $^{(**)}$ The Simons Foundation grant No. 360242.}}
Given two positive integers 
$ b $, $d $ 
and a 
$ bd \times bd $
matrix 
$ A $, 
we define the 
$ (b, d)$-partial transpose of 
$ A $, denoted here by
$ A^{\Gamma (b, d)} $,
as follows.
We consider 
$ A $ 
as a 
$ b\times b $
matrix, with each entry a block of size
$ d \times d $;
then we transpose the entries of each block, while leaving the block in place.  The study of partial transposes is motivated by their relevance in Quantum Information Theory, where they are related to the concept of entanglement and to the ``distillability conjecture'' (see \cite{horo1}, \cite{horo2}, also \cite{aubrun}). Also relevant to the general theory and to the present paper is the family of states with a Positive Partial Transpose (PPT states); as stated in \cite{aubrun}, non-PPT states are necessarily entangled (see \cite{peres}) thus giving a simple test to detect entanglement.

 In \cite{aubrun} it was shown that the limit distribution of the 
 $ (d, d)$-partial transpose of a Wishart random matrix is a shifted semicircular, and supported on the positive half of the real axis when the shape parameter is larger than $4$. The result was extended in \cite{wishart} to more general partial transposes of Wishart random matrices. Furthermore, in \cite{mpsz1} it was shown that the limit distribution of the 
 $(b, d)$-partial transpose of a Haar unitary, when both $b $ and $ d $ tend to infinity, is circular. Here we shall show that these are particular cases of a more general result concerning partial transposes of unitarily invariant random matrices. More precisely, Theorem \ref{thm:limdistrib} states that if 
 $ (b_N)_N , (d_N)_N $ are two increasing unbounded sequences of positive integers and 
 $\big( A_{1, N}, \dots, A_{R, N} \big)_N $ 
 is an ensemble of  unitarily invariant $ R $-tuples of random matrices with the bounded cumulants property, and for which the  $\ast$-cumulants of rank one and two converge, each 
 $ A_{k, N} $
 of size 
 $ b_N \times d_N $,  
 then the sequence of  
$ (b_N, d_N)$-partial transposes of 
$ A_{k, N}$
has a limit $\ast$-distribution which is obtained by erasing all free $\ast$-cumulants of order higher than 2 and taking the limits for the free cumulants of lower order.  

In \cite{mingo-popa-transpose} it was shown that transposing a matrix can make it asymptotically freely independent from itself. Many more interesting examples were provided  by Male \cite{male}. Prior to \cite{mingo-popa-transpose}, all examples of asymptotic freeness required independence of the entries. This led us to consider the effect of partial transposes.  It is known (see \cite{arizmendi}, \cite{banica-nechita}, \cite{wishart}, \cite{wishart2}, \cite{mpsz1}) that partial transposes of Wishart and Haar unitary random matrices are, when the size of blocks tends to infinity, asymptotically freely independent from the initial matrix. In another direction, Yin and Zhao \cite{yz} showed the strong convergence of partial transposes. The main result of this paper (Theorem \ref{thm:free}) shows that asymptotic free independence holds for the more general setting of unitarily invariant ensembles of random matrices with the bounded cumulants property.  

Besides the Introduction, the paper is organized in 3 more sections. Section 2 (Preliminaries) reviews some results from \cite{collins-sniady} (basics of unitary Weingarten calculus) and \cite{mpsz1} (asymptotic behavior of permuted Haar unitaries). Section 3 (Notations and Auxiliary Lemmata) presents several new notations and a suite of several technical results. Finally, Section 4 (Main Results) contains the main results of the paper, Theorems \ref{thm:limdistrib} and \ref{thm:free} mentioned above.

\section{Preliminaries}

\subsection{Review of the Unitary Weingarten Calculus}\label{weingarten}
One of the main results in \cite{collins-sniady} is the following. If 
$ U $ 
is a 
$ M \times M $
Haar unitary matrix, and 
$ u_{i, j} $ 
is its 
$ (i, j) $-entry, then
	\begin{align}\label{eqn:Weingarten}
	\E\left(
	u_{i_1,j_1}u_{i_2,j_2}
	\cdots u_{i_m,j_m}\overline{u_{i^\prime_1,j^\prime_1}}\, 
	\cdots\overline{u_{i^\prime_m,j^\prime_m}}\right)
	=
	\sum_{\sigma,\tau\in S_{m}}\Big(\prod_{k=1}^n\delta^{i_k}_{i^\prime_{\sigma(k)}}\delta^{j_k}_{j^\prime_{\tau(k)}}\Big) 
	\mathrm{Wg}_M (\sigma^{-1}\tau),
	\end{align}
where
$ \mathrm{Wg}_M $
is the unitary Weingarten function.
Furthermore, for
$\sigma\in S(m) $ 
with cycle decomposition 
$\sigma=\mathcal{C}_1\cdots \mathcal{C}_{\#( \sigma)}$,
and
$ Cat_r $ 
denoting the $r $-th Catalan number, denoting
\begin{align*}
C(\sigma)= \prod_{i=1}^{\#( \sigma)} (-1)^{|\mathcal{C}_i|-1} Cat_{|\mathcal{C}_i|-1}
\end{align*}
we have that
	\begin{align}\label{eqn:Wg:O}
		\mathrm{Wg}_M ( \sigma )=
	M^{-2m+\#(\sigma)}\cdot C(\sigma)+O(M^{-2m+\#(\sigma) -2}).
	\end{align}

 As described in \cite{mpsz1}, the equation above can be easily re-written for the framework when factors of the type 
 $ u_{i_s j_s} $ 
 and 
 $ \overline{u_{i_s, j_s}} $ 
 are not segregated in the product. 
For 
$ n $ 
a positive integer, we shall denote by
$ [ n ] $ 
the ordered set 
$ \{ 1, 2, \dots, n \}$.
and denote by
$ \cP_2( 2m) $
the set of pair-partitions on 
$ [ 2m ]$.
Given a map 
$ \varepsilon : [2 m ] \rightarrow \{ 1, \ast \} $,
we shall write 
$ \varepsilon_j $ 
for
$ \varepsilon(j) $
and we shall denote by 
$\cP_2^{\varepsilon} (2m) $ 
the set (possibly void) of pair-partitions on
$ [2 m ] $
that connect elements with different values of 
$\varepsilon$:
\[ \cP_2^{\varepsilon} (2m)  = \{ p \in \cP_2(2m):\ \varepsilon_{k} \neq \varepsilon_{p(k)} \textrm{ for all } k \in [ 2m ] \}. \]
With the notations above, Equation (\ref{eqn:Weingarten}) becomes
\begin{align}\label{eqn:Wg:2}
\E \big( 
u_{i_1,j_1}^{\epsilon_{1}}
u_{i_2,j_2}^{\epsilon_{2}}
\cdots u_{i_{2m},j_{2m}}^{\epsilon_{2m}}
\big)
=
\sum_{p,q\in \mathcal{P}^\epsilon_{2}(2m)}
\bigg(
\prod_{k=1}^{2m}\delta^{i_k}_{i_{p(k)}}\delta^{j_k}_{j_{q(k)}}
\bigg)
\mathrm{Wg}_M( p , q ),
\end{align}
where the function
$ \Wg_M ({}\cdot{},{}\cdot{}) $ 
is obtained from the unitary Weingarten function as follows. For
$ p, q \in \cP_2^\varepsilon ( 2m ) $, 
denote by
$ p \vee q $ 
the supremum of 
$ p $
and 
$ q $
in the lattice of the partitions of 
$ [ 2m] $ 
(see \cite{nica-speicher} for more details).
More precisely, if  
$ p \vee q $ 
has the block decomposition
$  B_1, B_2, \cdots, B_{ | p \vee q |} $,
then 
each of the blocks 
has an even number of elements
$ \{ a_1, a_2, \dots, a_{2s}\}$
with
$ 	a_2=p(a_1),\,a_3=q(a_2),\,a_4=p(a_2),\ldots,a_1=q(a_{2s}) $.
Finally, we define 
 $ \mathrm{Wg}_M (p, q) = \mathrm{Wg}_M ( p \vee q ) $,
  where the partition $p\vee q $ is identified with the permutation with cycles
 	 $B_1, B_2, \cdots B_{|p \vee q |} $,
 	  as described above.
  In particular, Equation (\ref{eqn:Wg:2}) gives that
  \begin{align}\label{w-mult}
\mathrm{Wg}_M(p,q)=M^{-2m+|p \vee q |}
C(p \vee q)
+
O(M^{-2m+|p \vee q |-2}).
\end{align}

\subsection{Finding the leading order}\label{prop:3:1}
In this subsection we will recall a technical result from \cite{mpsz1} that shall be used for second part of the paper. 
 
 First, we need to introduce several notations:
 \begin{enumerate}
 	\item[$\mybullet$] for
 $ \sigma $ 
a permutation on the the set 
$ [ M ]^2 = \{ (i, j):\ 1 \leq i, j \leq M \} $,  
and 
$ A $ 
 a 
  $ M \times M $
 matrix, we define 
 $ A^\sigma $
 as the 
 $ M \times M $
 matrix whose 
 $ (i, j)$-entry is the $ \sigma(i,j)$ -entry of $ A $, 
 i.e. 
 $ [ A^\sigma ]_{i, j} = [A ]_{\sigma(i, j)} $;

 \item[$\mybullet$] for
  $ \varepsilon \in \{ 1, \ast \} $
   and 
 $ a $ 
 a complex random variable, and
 $ (i, j) \in [ M]^2 $, ,
  we denote
  \begin{align*}
  a^{(\varepsilon)}=\begin{cases}
  a \mbox{ if } \varepsilon =1,\\
  \overline{a} \mbox{ if } \varepsilon =\ast\\
  \end{cases}
  \hspace{1cm} \textrm{  and } \hspace{1cm}
  \varepsilon (i,j)=\begin{cases}
  (i,j) \mbox{ if } \varepsilon =1,\\
  (j,i) \mbox{ if } \varepsilon =\ast\\
  \end{cases};
  \end{align*}

\item[$\mybullet$]
for a $M \times M$ matrix $A$ we let $\Tr(A)$ be the sum of the diagonal entries of $A$, the trace of $A$, and $\tr(A) = M^{-1} \Tr(A)$, the \textit{normalized trace} of $A$.

 \end{enumerate}

 For
$ \sigma: [ 2m] \rightarrow \mathcal{S}([ M]^2) $ 
and
$ \varepsilon: [ 2m] \rightarrow \{1, \ast \} $,
we shall write
$ \sigma_j $,
respectively
$ \varepsilon_j $
for 
$ \sigma(j) $
and 
$ \varepsilon(j) $.
 With the notations from \ref{weingarten} we have that
 \begin{multline*}
 \E\big(\tr \big(
  (U^{ \sigma_1})^{\varepsilon_1} 
  (U^{ \sigma_2})^{\varepsilon_2}
  \cdots
 (U^{ \sigma_{2m}})^{\varepsilon_{2m}} 
   \big)\big)\\
    = 
   {M^{-1}} \kern-1em\sum_{ i_1, \dots, i_{2m}=1}^N \kern-1em 
   \E \big(
u^{( \varepsilon_1)}_{ \sigma_1 \circ \varepsilon_1 ( i_1, i_2 ) }
u^{( \varepsilon_2)}_{ \sigma_2 \circ \varepsilon_2 ( i_2, i_3 ) }
\cdots
u^{( \varepsilon_{2m})}_{ \sigma_{2m} \circ \varepsilon_{2m} ( i_{2m}, i_1 ) }
   \big)
   =  
   {M^{-1}} \kern-1em\sum_{ i_1, \dots, i_{2m}=1}^M \kern-1em
 \E \big(  u^{(\varepsilon_1)}_{ k_1, l_1} \cdots 
 u^{(\varepsilon_{2m})}_{ k_{2m}, l_{2m}} \big)
 \end{multline*}
where for each
$ s$
 we define $k_s$ and $l_s$ by the equation
$ (k_s, l_s) = \sigma_s (\varepsilon_s (i_s, i_{s+1}) )
$ , 
 with the identification
$ i_{2m+1} = i_1 $. 

Thus, Equation (\ref{eqn:Wg:2}) gives that 
\begin{align}
 \E \circ \tr \big(
(U^{ \sigma_1})^{\varepsilon_1} 
(U^{ \sigma_2})^{\varepsilon_2}
\cdots
(U^{ \sigma_{2m}})^{\varepsilon_{2m}} 
\big)
= \sum_{p, q \in \cP_2^{\varepsilon}(2m) } \mathcal{V} (\sigma, \varepsilon, M, p, q)
\end{align}
 where
\begin{align}\label{v:pq}
 \mathcal{V} (\sigma, \varepsilon, M, p, q)
 & =
 \frac{1}{M}
 \mathrm{Wg}_M( p , q )
  \sum_{i_1, \dots, i_{2m}=1}^M
 \sum_{j_1, \dots, j_{2m}=1}^M 
   \prod_{s=1}^{2m}
    \delta_{k_s}^{k_{p(s)}} \delta_{l_s}^{l_{q(s)}}\delta_{j_s}^{i_{s+1}}\\
 & =
 \frac{1}{M}
 \mathrm{Wg}_M( p , q ) \cdot 
 | \mathcal{A}^{(p, q)}_{\sigma, \varepsilon, M}|
\nonumber
\end{align}
where
$ (k_s, l_s)$ are as above.
Let 
 $ \mathcal{A}^{(p, q)}_{\sigma, \varepsilon, M} $
be the set 
 \begin{align*}
 \mathcal{A}^{(p, q)}_{\sigma, \varepsilon, M}=
\big\{ 
(i_1, j_1, \dots, i_{2m}, j_{2m}) \in [ M]^{4m}\mid \ j_s= i_{s+1},
  k_s = k_{p(s)},
l_s = l_{q(s)} \textrm{ for all } s \in [ 2m ]
\big\}.
 \end{align*}
 
 A key result obtained in \cite{mpsz1} (Proposition 3.1) is the following consequence of Equations (\ref{eqn:Weingarten}) and (\ref{w-mult}).
\begin{lemma}\label{lemma:prop:3:1}
	With the notations from above, we have that 
$\mathcal{V}( \sigma, \varepsilon , M, p, q)
= O(M^{-1}) $
unless the following conditions are satisfied, case in which 
$\mathcal{V}( \sigma, \varepsilon, M, p, q)
= O(M^{0}) $:
\begin{enumerate}
	\item[$(i)$] $ p\vee q $ is non-crossing.
	\item[$(ii)$] for any block 
	$\{ a_1, \dots, a_r \} $ 
	of 
	$ p \vee q $ 
	with 
	$ a_1 < a_2 < \dots < a_r $ and under the identifications 
	$ a_{r+1} = a_1 $ and $ a_0 = a_r $
	 we have that, for each 
	$ s =1, \dots, r$,
	\begin{align}\label{as}
	\{ p(a_s), q(a_s)\} = \{ a_{s-1}, a_{s+1}\}.
	\end{align}
$($in particular, if $(\ref{as})$ holds true, then \marginpar{parts}
$ \varepsilon_{s} \neq \varepsilon_{s+1} $
for all 
$ s =1, \dots, r $.$)$
\end{enumerate}	
\end{lemma}

Will shall briefly utilize another technical result from \cite{mpsz1}, stated below (a particular case of Lemma 3.2 $(iv)$):
\begin{lemma}\label{lemma:segment}
 Suppose that 
$ p\vee q $ has a block,
$ B$,
 of the form 
$(s+1, s+2, \dots, s+t) $
and denote by
$ \iota_B $
the map 
$ (i_1, j_1, \dots, i_{2m}, j_{2m}) 
\mapsto
(i_{s+1}, j_{s+1}, \dots, i_{s+t}, j_{s+t})$.

If 
$ \big| \iota_B ( \mathcal{A}^{(p, q)}_{\sigma, \varepsilon, M})\big|
= o(M^0)$
then 
$\mathcal{V}( \sigma, \varepsilon, M, p, q)
= O(M^{0})$.
\end{lemma}

 \section{Auxiliary lemmata}\label{aux}
 
 Let $ b $ and $ d $ be two positive integers and let
 $  M = b \cdot d $.
 We define 
 $(b, d)$ partial transpose 
 $ \Gamma_{b, d} $
 as follows (see also the intuitive definition from the Introduction). If we consider the bijection
 $\psi : [b]\times [d] \rightarrow [M] $ 
 given by
 $ \psi(\alpha, \beta) = (\alpha -1)d + \beta $,
 then the map 
  $ \Gamma_{b, d} $ 
  is the permutation of $ [ M]^2 $
  given by (here 
  $ \alpha_1, \alpha_2 \in [ b] $ 
  and
  $ \beta_1, \beta_2 \in [ d ] $):
 \[ 
  \Gamma_{b, d}\big( \psi(\alpha_1, \beta_1), \psi(\alpha_2, \beta_2)\big)
  = \big( \psi( \alpha_1, \beta_2), \psi(\alpha_2, \beta_1) \big).
  \]
  The map 
  $ \Ltr_{b, d} $ 
(``the left $ (b, d)$- partial transpose'') is given by
$ \Ltr_{b, d}= \mathrm{T} \circ \Gamma_{b, d} $, where T denotes the transpose.
 Since all the partial transposes in this section will be of the same type, we will not use the notations 
 $ \Gamma_{b, d} $ 
 and 
 $ \Ltr_{b, d}$,
 but the shorter 
 $ \Gamma $ and $ \Ltr $.
 
 For $ A $ a matrix of size 
 $ M \times M$ 
 again the simplify the notations, we shall use the, as in \cite{wishart}, the following  notation (here
 $ \theta, \eta \in \{ -1, 1\}$):
 \begin{align}\label{notation:2-1}
 A^{ (\theta, \eta)} = \left\{ 
 \begin{array}{ll}
 A & \textrm{ if } (\theta, \eta) = (1, 1)\\
 A^{\ltr} & \textrm{ if } (\theta, \eta)= ( -1, 1)\\
 A^\rtr & \textrm{ if } (\theta, \eta) = ( 1, -1)\\
 A^\ftr & \textrm{ if } (\theta, \eta)= ( -1, -1).
 \end{array}
 \right.
 \end{align} 
 We shall also use the shorthand notation
$ [ A ]_{\alpha_i, \beta_i, \alpha_{j}, \beta_{j}} $
for the
$( \psi(\alpha_i, \beta_i), \psi(\alpha_{j}, \beta_{j}))$-entry of the matrix 
$ A $, 
i.e. for
$ [ A]_{\psi(\alpha_i, \beta_i), \psi(\alpha_{j}, \beta_{j})}$. 

Given a map 
$ \varphi : [m] \rightarrow \{ 1, -1\} $ 
we will denote by 
$ \widehat{\varphi}$ 
the permutation from 
$ \mathcal{S}(2m) $
given by
 \begin{align*}
\widehat{\varphi}(2s-1)= \left\{
\begin{array}{ll}
2s-1 & \textrm{ if } \varphi(s) = 1\\
2s & \textrm{ if } \varphi(s) = -1,
\end{array} \right.
\textrm{ respectively }
\widehat{\varphi}(2s)= \left\{
\begin{array}{ll}
2s & \textrm{ if } \varphi(s) = 1\\
2s-1 & \textrm{ if } \varphi(s) = -1.
\end{array} \right. .
\end{align*}

 Assume that 
 $ m $ 
 is a positive integer, 
 $ A $, as above, is  a $ M \times M $ random  matrix, and 
 $ \theta, \eta: [m]\rightarrow \{1, -1\} $.
  For each
 $ s \in [ m ] $,
 with the notations above
 we have that
 \begin{align*}
 [A^{(\theta(s), \eta(s))}]_{
  \alpha_{2s-1}, \beta_{2s-1}, \alpha_{2s}, \beta_{2s}} = 
[A]_{\alpha_{\widehat{\theta}( 2s-1)}, 
	\beta_{ \widehat{\eta}(2s-1)}, 
	\alpha_{\widehat{\theta}(2s)},
	\beta_{\widehat{\eta}(2s)}}.
 \end{align*}
 
 In particular, if 
 $ U $
 is a unitary Haar random matrix of same size as A,
 we have that
 \begin{align}\label{entry}
 [(UAU^\ast)^{(\theta(s), \eta(s) )}&]_{
 	 \alpha_{2s-1}, \beta_{2s-1}, \alpha_{2s}, \beta_{2s}}\\
 & \ \  \ \ \ = \sum_{k, k^\prime =1}^M 
 [U]_{ \psi (\alpha_{\widehat{\theta}(2s-1)}, \beta_{\widehat{\eta}(2s-1) } ), k }
 [ A]_{k, k^\prime}
 \overline{[U]_{ \psi (\alpha_{\widehat{\theta}(2s)}, \beta_{\widehat{\eta}(2s) } ), k^\prime } }.\nonumber
 \end{align}

\begin{definition}
 A $ m $-tuple 
 $\big( A_1, A_2, \dots, A_m)$
 of $ M \times M $
 random matrices is said to be 
 \emph{unitarily invariant} if for any
 $ M \times M $ 
 unitary matrix 
 $ U $, 
 any $ r_1, \dots, r_k \in [m]$,
 and any $i_1, j_1, \dots, i_s, j_s \in [ M ]$
  we have that
  \[ \E (a_{i_1, j_1}^{(r_1)} a_{i_2, j_2}^{(r_2)}\cdots a_{i_s j_s}^{(r_s)} ) 
  = \E (b_{i_1, j_1}^{(r_1)}b_{i_2, j_2}^{(r_2)} \cdots b_{i_s j_s}^{(r_s)})
  \]
  where
  $a_{i,j}^{(k)} $
  is the 
  $(i, j)$-entry of 
  $ A_k $
  and 
  $ b_{i, j}^{(k)} $
  is the 
  $ (i, j)$-entry of 
  $ UA_kU^\ast$.
\end{definition} 

 Let 
 $ \big(T_1, T_2, \dots, T_m ) $ 
 be a unitarily invariant $m$-tuple of
  $M \times M $
   random matrices (where 
   $ M = b \cdot d $)
    and let 
 $ U $ 
 be a 
 $ M \times M $
 Haar unitary random matrix with entries independent from the entries of the $m$-tuple 
 $ \big(T_1, T_2, \dots, T_m ) $. 
 With the notations from above we have that
 \begin{multline*}
\E \big(\tr \big( 
 T_1^{( \theta(1), \eta(1))}
\cdot
T_2^{( \theta(2), \eta(2))}
\cdots 
( T_m^{( \theta(m), \eta(m))}
\big)\big)\\
= 
\E \big(\tr \big( \,
( U T_1 U^\ast)^{( \theta(1), \eta(1))}
\cdot
( U T_2 U^\ast)^{( \theta(2), \eta(2))}
\cdots
( U T_m U^\ast)^{( \theta(m), \eta(m))} 
\big)\big)\\
 =  \frac{1}{M}
\sum_{\substack{ \vec{\alpha}\in [b]^{2m} \\
		\vec{\beta}\in [d]^{2m} \\
		\vec{\gamma}\in [M]^{2m} } }
\E \Big( \prod_{s=1}^m
[ U]_{
	\psi(\alpha_{\widehat{\theta}(2s-1)}, 
	\beta_{\widehat{\eta}(2s-1) }), \gamma_{2s-1}}
[ T_s]_{ \gamma_{2s-1}, \gamma_{2s}} 
\overline{ [ U]_{\psi( \alpha_{\widehat{\theta}(2s) },
		\beta_{ \widehat{\eta}(2s)}), \gamma_{2s}}
}
\delta^{\alpha_{2s}}_{\alpha_{2s+1}}
\delta^{\beta_{2s}}_{\beta_{2s+1}}
\Big),
\end{multline*} 
  where
 $ \vec{\alpha} =(\alpha_1, \dots, \alpha_{2m})$, 
 $\vec{\beta} =(\beta_1, \dots, \beta_{2m})$
 and
 $\vec{\gamma} = (\gamma_1, \gamma_2, \dots, \gamma_{2m})$.
 
  When
 $ \varepsilon: [2m] \rightarrow \{ 1, \ast \}$
 via
 $ \displaystyle
 \varepsilon(s) = \left\{ 
 \begin{array}{ll}
 1 & \textrm{ if $ s $ is odd} \\
 \ast & \textrm{ if $ s $ is even} 
 \end{array}
 \right.
 $, 
 Equation  (\ref{eqn:Wg:2}) gives that
 \begin{multline*}
 \E \big(\tr\big( 
 T_1^{( \theta(1), \eta(1))}
 \cdot
T_2^{( \theta(2), \eta(2))}
 \cdots 
T_m^{( \theta(m), \eta(m))}   
 \big)\big)\\
 =  \frac{1}{M}
 \sum_{\substack{ \vec{\alpha}\in [b]^{2m} \\
 		\vec{\beta}\in [d]^{2m} \\
 		\vec{\gamma}\in [M]^{2m} } }
 \sum_{p, q \in \cP_2^{\varepsilon}(2m) } 
 \mathrm{Wg}_M (p, q ) 	
 \E \Big( 
 \prod_{s =1}^m 
 [ T_s]_{ \gamma_{2s-1}, \gamma_{2s}} \Big)
 \delta^{\gamma_{2s}}_{\gamma_{ q(2s)}}
 \delta^{
 	\psi(\alpha_{\widehat{\theta}(2s)}, \beta_{\widehat{\eta}(2s) })
 }_{
 	\psi(\alpha_{\widehat{\theta}\circ p(2s)}, 
 	\beta_{\widehat{\eta}\circ p(2s)})
 } 
 \delta^{\alpha_{2s}}_{\alpha_{2s+1}}
 \delta^{\beta_{2s}}_{\beta_{2s+1}}
.
 \end{multline*} 
 
In order to simply the writing in the rest of this section, we shall recall some notation from \cite{mingo-popa-transpose}. 
Assume that 
$ A_1, A_2, \dots, A_m $
are 
$ M \times M $ 
matrices.
If
$ \gamma $
is a permutation on the set 
$ [ m ] $
with cycle decomposition
$ \sigma = \mathcal{C}_1 \cdot \mathcal{C}_2 \cdots\mathcal{C}_r $
with 
$ \mathcal{C}_k = ( i_{(k, 1)}, i_{(k, 2)}, \dots, i_{(k, l(k))} )$,
then we let
\begin{align*}
\Tr_{\sigma} ( A_1, A_2, \dots, A_m )&  = \prod_{k=1}^r
\Tr(A_{i_{(k, 1)}}A_{i_{(k, 2)}} \cdots A_{i_{(k, l(k))}} ) \\
&= \sum_{ i_1, \dots, i_m =1}^M [A_1]_{i_1, i_{\sigma(1)}} [ A_2]_{i_2, i_{\sigma(2)}}
\cdots [A_m]_{i_m, i_{\sigma(m)}}.\
\end{align*}
and for the normalized trace we let 
 \begin{align*}
 \tr_{\sigma} ( A_1, A_2, \dots, A_m )&  = \prod_{k=1}^r
 \tr(A_{i_{(k, 1)}}A_{i_{(k, 2)}} \cdots A_{(i_{(k, l(k))}} ) \\
 &= M^{- \# (\sigma)}\sum_{ i_1, \dots, i_m =1}^M [A_1]_{i_1, i_{\sigma(1)}} [ A_2]_{i_2, i_{\sigma(2)}}
 \cdots [A_m]_{i_m, i_{\sigma(m)}}.
 \end{align*}

 The pair-partition 
 $ q \in \cP_2^{\varepsilon} (2m) $
 induces a unique permutation 
 $ \widehat{q} \in \mathcal{S}(m) $
 via the relation
 $ 2\widehat{q}(s) -1  =  q(2s)$
 (note that since
 $ q \in \cP_2^{\varepsilon}(2m ) $,
 the value $q(2s) $ is odd for all $ s $).
 Denoting 
 $ t_s = \gamma_{2s-1} $, 
 and
 $ \vec{t} =(t_1, \dots, t_{m})$
 we have that 
 \begin{align*}
 \sum_{ \vec{\gamma} \in [M]^{2m} }
 \prod_{s =1}^m 
 [ T_s]_{ \gamma_{2s-1}, \gamma_{2s}}
 \delta^{\gamma_{2s}}_{\gamma_{ q(2s)}} 
 = \sum_{\vec{t}\in [ M]^m }
 \prod_{s =1}^m 
 [ T_s]_{ t_s, t_{\widehat{q}(s)} }
 = \Tr_{\widehat{q}} \big(T_1, T_2, \dots, T_m \big).
 \end{align*}
 
 Finally, note also that 
 $ \delta^{
 	\psi(\alpha_{\widehat{\theta}(2s)}, \beta_{\widehat{\eta}(2s) })
 }_{
 	\psi(\alpha_{\widehat{\theta}\circ p(2s)}, 
 	\beta_{\widehat{\eta}\circ p(2s)})
 }  
 =
 \delta^{
 	\alpha_{\widehat{\theta}(2s)}
 }_{
 	\alpha_{\widehat{\theta}\circ p(2s)}
 }
 \cdot
 \delta^{\beta_{\widehat{\eta}(2s) }}_{\beta_{\widehat{\eta}\circ p(2s)}}
 $ 
 and that
 \begin{align*}
 \prod_{s=1}^m 
 \delta^{
 	\alpha_{\widehat{\theta}(2s)}
 }_{
 	\alpha_{\widehat{\theta}\circ p(2s)}
 }
 =
 \prod_{s=1}^{2m}
 \delta^{
 	\alpha_{\widehat{\theta}(s)}
 }_{
 	\alpha_{\widehat{\theta}\circ p(s)}
 }
 =
 \prod_{s=1}^{2m}
 \delta^{
 	\alpha_{s}
 }_{
 	\alpha_{\widehat{\theta}\circ p \circ \widehat{\theta}^{-1} (s)}
 }.
 \end{align*}
 
  Therefore, denoting
 \begin{align}\label{it:W}
 \mathcal{W}_{\theta, \eta} 
 (p, q, T_1, \dots, T_m) = 
 \frac{1}{M}
 \mathrm{Wg}_M &(p, q ) 
 \cdot 
 \big[\E \circ \Tr_{\widehat{q}}\big(T_1,  \dots , T_m \big)\big] \\
 &\cdot
 \sum_{\substack{ \vec{\alpha}\in [b]^{2m} \\
 		\vec{\beta}\in [d]^{2m}}}
 \prod_{s=1}^{2m} 
 \delta^{
 	\alpha_{s}
 }_{
 	\alpha_{\widehat{\theta}\circ p \circ \widehat{\theta}^{-1} (s)}
 }
 \delta^{
 	\beta_{s}
 }_{
 	\beta_{\widehat{\eta}\circ p \circ \widehat{\eta}^{-1} (s)}
 }  	
 \prod_{s=1}^{m}\delta^{\alpha_{2s}}_{\alpha_{2s+1}}
 \delta^{\beta_{2s}}_{\beta_{2s+1}}\nonumber
 \end{align}
 we have that
\begin{align*}
\E \circ \tr\big( &
 T_1^{( \theta(1), \eta(1))}
\cdot
T_2^{( \theta(2), \eta(2)}
\cdots 
T_m^{( \theta(m), \eta(m))}
   \big)
=
\sum_{p, q \in \cP_2^{\varepsilon}(2m) } 
\mathcal{W}_{\theta, \eta}
(p, q, T_1, \dots, T_m).
\end{align*}

 If for all $ k $ we have that 
$ T_k = T $ 
(some 
$ M \times M $
random matrix), we shall write
$ \mathcal{W}_{\theta, \eta} (p, q, T)$
for
$\mathcal{W}_{\theta, \eta}
(p, q, T_1, T_2, \dots, T_m)$.
For example, if  
$ I_M $ 
denotes the identity $ M \times M $ matrix, we shall write
 $ \mathcal{W}_{\theta, \eta} (p, q, I_M)$
 for 
$ \mathcal{W}_{\theta, \eta} (p, q, I_M, \dots, I_M)$.

In order to  state the next result, we need to introduce one more notation.
We define the maps 
$ \sigma_{\theta}: [ 2m] \rightarrow \mathcal{S}( [b]^2 )$,
respectively
$ \sigma_{\eta}: [2m] \rightarrow \mathcal{S}( [d ]^2)$
via
\[
\sigma_{\theta}(2s-1) = \sigma_{\theta}(2s) = \left\{ 
\begin{array}{ll}
\textit{id} & \textrm{ if } \theta(s) = 1\\
\top & \textrm{if } \theta(s) = -1
\end{array}
\right.\textrm{, and }
\]
\[
\sigma_{\eta}(2s-1) = \sigma_{\eta}(2s) = \left\{ 
\begin{array}{ll}
\textit{id} & \textrm{ if } \eta(s) = 1\\
\top & \textrm{if } \eta(s) = -1
\end{array}
\right..
\]

\begin{lemma}\label{lemma:W:V}
	With the notations from Sections \ref{weingarten} and \ref{prop:3:1},
	we have that 
	\begin{multline*}
	\mathcal{W}_{\theta, \eta} (p, q, I_M) 
	= \frac{1}{C(p, q) }
	\cdot 
	\mathcal{V} (\sigma_{\theta}, \varepsilon\circ \widehat{\theta}, b, 
	\widehat{\theta} \circ p \circ \widehat{\theta}^{-1},
	\widehat{\theta} \circ q \circ \widehat{\theta}^{-1}  ) \\
	\cdot 
	\mathcal{V}(\sigma_{\eta}, \varepsilon\circ \widehat{\eta},
	d,\widehat{\eta} \circ p \circ \widehat{\eta}^{-1},
	\widehat{\eta} \circ q \circ \widehat{\eta}^{-1} )
	+ o(M^0).
	\end{multline*}
\end{lemma}

\begin{proof}
	First, note that if
	$ M = b\cdot d $ with $ b, d \rightarrow \infty $,
	then, for any 
	$ p, q \in \cP_2^{\varepsilon}(2m)$, 
	Equation (\ref{w-mult}) gives that
	\begin{align*}
	\textrm{Wg}_b(p, q) \cdot \textrm{Wg}_d (p, q) =
	C(p, q)\cdot  \textrm{Wg}_M(p, q)
	+ o(M^{-2m+| p \vee q |}).
	\end{align*}
	
	Next, remark that, for each
	$ s \in [ m ] $, 
	the set 
	$\{ 2s-1, 2s\} $
	is invariant under both maps 
	$ \widehat{\theta}$
	and 
	$ \widehat{\eta}$. 
	Therefore, denoting 
	$ \gamma_{s} = \psi( \gamma^\prime_{s}, \gamma^{\prime\prime}_{s}) $,
	and
	$ \xi^\prime_s = \gamma^\prime_{\widehat{\theta}}(s)$,
	$ \xi^{\prime\prime}_s = \gamma^{\prime\prime}_{\widehat{\eta}}(s)$
	we have that
	\begin{align*}\lefteqn{
	\prod_{s=1}^m[I_M]_{\gamma_{2s-1}, \gamma_{2s}}
	\delta^{\gamma_{2s}}_{\gamma_{q(2s )} } }\\
	&= 
	\prod_{s=1}^{m}
	\delta_{\gamma^\prime_{\widehat{\theta}(2s-1) } }^{\gamma^\prime_{\widehat{\theta}(2s)} }
	\delta^{\gamma^\prime_{2s}}_{\gamma^\prime_{q(2s )} }
	\delta_{\gamma^{\prime\prime}_{\widehat{\eta}(2s-1) } }^{\gamma^{\prime\prime}_{\widehat{\eta}(2s)} }	
	\delta^{\gamma^{\prime\prime}_{2s}}_{\gamma^{\prime\prime}_{q(2s )} }
	=
	\prod_{s=1}^m 
	\delta_{\xi_{2s-1}^\prime }^{\xi_{2s}^\prime}
	\delta_{\xi_{2s-1}^{\prime\prime} }^{\xi_{2s}^{\prime\prime}}
	\prod_{s=1}^{2m} 
	\delta^{\xi^\prime_s}_{\xi^\prime_{\widehat{\theta} \circ q \circ\widehat{\theta}^{-1} (s) } }
	\delta^{\xi^{\prime\prime}_s}_{\xi^{\prime\prime}_{\widehat{\theta} \circ q \circ\widehat{\theta}^{-1} (s) } }.
	\end{align*}
	Equation (\ref{it:W}) gives then 
	\begin{align}\label{w:2buc}
	\mathcal{W}_{\theta, \eta} (p, q, I_M) = 
	\frac{1}{M}
	\textrm{Wg}_M &(p, q) \cdot
	\big(  \sum_{ \vec{\alpha}, \vec{\xi^\prime} \in [ b ]^{2m}  }
	\prod_{s=1}^{m}
	\delta^{\xi^\prime_{2s}}_{\xi^\prime_{2s-1} }
	\delta^{\alpha_{2s}}_{\alpha_{2s+1}}
	\prod_{s=1}^{2m} 
	\delta^{
		\alpha_{s}
	}_{
		\alpha_{\widehat{\theta}\circ p \circ \widehat{\theta}^{-1} (s)}}
	\delta^{\xi^\prime_s}_{\xi^\prime_{\widehat{\theta} \circ q \circ \widehat{\theta}^{-1} (s) } }
	\big)\\
	&\cdot \big(
	\sum_{ \vec{\beta}, \vec{\xi^{\prime\prime}} \in [ d ]^{2m}  }
	\prod_{s=1}^{m}
	\delta^{\xi^{\prime\prime}_{2s}}_{\xi^{\prime\prime}_{2s-1} }
	\delta^{\beta_{2s}}_{\beta_{2s+1}}
	\prod_{s=1}^{2m} 
	\delta^{
		\beta_{s}
	}_{
		\beta_{\widehat{\eta}\circ p \circ \widehat{\eta}^{-1} (s)}}
	\delta^{\xi^{\prime\prime}_s}_{\xi^{\prime\prime}_{\widehat{\eta} \circ q \circ\widehat{\eta}^{-1} (s) } }
	\big).
	\nonumber
	\end{align}
	
	On the other hand, with the notations from Section \ref{prop:3:1}, if, for each
	$ s \in [ m ] $, 
	we put
	$ i_{2s-1} = \alpha_{2s-1} $, 
	$ j_{2s-1} = \xi^\prime_{2s-1}$,
	$ i_{2s}= \xi^\prime_{2s} $
	and
	$ j_{2s} = \alpha_{2s} $,
	then, for each
	$ t \in [2m] $
	we have that 
	\begin{align*}
	(k_t, l_t) = \sigma_{\theta}(t) \circ (\varepsilon \circ \widehat{\theta})(t)( i_t, j_t) = (\alpha_t, \xi^\prime_t).
	\end{align*}  
	therefore Equation (\ref{v:pq}) reads:
	\begin{align}\label{V:1}
	\mathcal{V} (\sigma_{\theta}, \varepsilon\circ \widehat{\theta}, b, 
	\widehat{\theta} \circ p \circ \widehat{\theta}^{-1},
	\widehat{\theta} \circ q \circ & \widehat{\theta}^{-1}  )
	=\\
	\frac{1}{b} \textrm{Wg}_b (\widehat{\theta} \circ p \circ \widehat{\theta}^{-1}, 
	\widehat{\theta} \circ q \circ  \widehat{\theta}^{-1}) 
	& \cdot \big(  \sum_{ \vec{\alpha}, \vec{\xi^\prime} \in [ b ]^{2m}  } 
	\prod_{s=1}^{m}
	\delta^{\xi^\prime_{2s}}_{\xi^\prime_{2s-1} }
	\delta^{\alpha_{2s}}_{\alpha_{2s+1}}
	\prod_{s=1}^{2m} 
	\delta^{
		\alpha_{s}
	}_{
		\alpha_{\widehat{\theta}\circ p \circ \widehat{\theta}^{-1} (s)}}
	\delta^{\xi^\prime_s}_{\xi^\prime_{\widehat{\theta} \circ q \widehat{\theta}^{-1} (s) } }
	\big).
	\nonumber
	\end{align}	
	Similarly, we have that
	\begin{align}\label{V:2}
	\mathcal{V} (\sigma_{\eta}, \varepsilon\circ \widehat{\eta}, d, 
	\widehat{\eta} \circ p \circ \widehat{\eta}^{-1},
	\widehat{\eta} \circ q & \circ  \widehat{\eta}^{-1}  )
	=\\
	\frac{1}{d} \textrm{Wg}_d (\widehat{\eta} \circ p \circ \widehat{\eta}^{-1}, 
	\widehat{\eta} \circ q \circ  \widehat{\eta}^{-1}) 
	& \cdot \big(  \sum_{ \vec{\beta}, \vec{\xi^{\prime\prime}} \in [ d ]^{2m}  } 
	\prod_{s=1}^{m}
	\delta^{\xi^{\prime\prime}_{2s}}_{\xi^{\prime\prime}_{2s-1} }
	\delta^{\beta_{2s}}_{\beta_{2s+1}}
	\prod_{s=1}^{2m} 
	\delta^{
		\beta_{s}
	}_{
		\beta_{\widehat{\eta}\circ p \circ \widehat{\eta}^{-1} (s)}}
	\delta^{\xi^{\prime\prime}_s}_{\xi^{\prime\prime}_{\widehat{\eta} \circ q \widehat{\eta}^{-1} (s) } }
	\big).
	\nonumber
	\end{align}
	
	Since 
	$ C(p, q) =
	C( \widehat{\theta} \circ p \circ \widehat{\theta}^{-1},\widehat{\theta} \circ q \circ \widehat{\theta}^{-1} )
	=
	C( \widehat{\eta} \circ p \circ \widehat{\eta}^{-1},\widehat{\eta} \circ q \circ \widehat{\eta}^{-1} )
	$
	the conclusion follows from Equations (\ref{w:2buc}), (\ref{V:1}), (\ref{V:2}) and Lemma \ref{lemma:prop:3:1}.
	
\end{proof}

\begin{corollary}\label{cor:2:3}
If 
$ \mathcal{W}_{\theta, \eta} (p, q, I_M) \neq o(M^0)$
then the following conditions are satisfied:
	\begin{enumerate}
		\item[$(i)$]the pairs 
		$ ( \widehat{\theta} \circ p \circ \widehat{\theta},
		\widehat{\theta} \circ q \circ \widehat{\theta},) $
		and 
		$ ( \widehat{\eta} \circ p \circ \widehat{\eta},
		\widehat{\eta} \circ q \circ \widehat{\eta},) $
		satisfy conditions $(i)$ and $(ii)$ from Lemma \ref{lemma:prop:3:1}.
		\item[$(ii)$] if 
		$ s $ 
		and 
		$ t $ 
		are in the same block of 
		$ p \vee q $ 
	then
	$ \theta(s) = \theta(t)$
	and
	$ \eta(s) = \eta(t) $.
	\end{enumerate} 
\end{corollary}
\begin{proof}
Part $(i)$ is an immediate consequence of  Lemmata \ref{lemma:prop:3:1} and \ref{lemma:W:V}.
For part $(ii)$, it suffices to show that if 
$ \mathcal{W}_{\theta, \eta} (p, q, I_M )  \neq o(M^0)$
then
$ \theta$ and 
$ \eta $
are constant on the blocks of 
$ \widehat{\theta} \circ p \circ \widehat{\theta} \vee
\widehat{\theta} \circ q \circ \widehat{\theta} $.
	
Assume that
$ \mathcal{W}_{\theta, \eta} (p, q, I_M)  \neq o(M^0)$.
Part $(i)$ gives that 
$ \widehat{\theta} \circ p \circ \widehat{\theta} \vee
\widehat{\theta} \circ q \circ \widehat{\theta} $
is non-crossing, hence it has a block
$ B $
with all elements consecutive. If 
$ \widehat{\theta} $ and $ \widehat{\eta} $ 
are constant on
$ B $, 
then the conclusion follows by erasing the block $ B $
and applying induction on 
$ m $. 
Suppose that either
$ \widehat{\theta} $ or $\widehat{ \eta} $ 
is not constant on $ B $. 
By construction, 
	$ \widehat{\theta} (2k-1) = \widehat{\theta}(2k) $ 
and 
	$ \widehat{\eta}(2k-1) = \widehat{\eta}(2k) $ 
	for all 
	$ k \in [ m ]$. 
	Hence there exist some 
	$ k \in [ m ] $ 
	with 
	$ 2k, 2k+1 \in B $
	and
	$ \widehat{\theta}(2k) \neq \widehat{\theta} (2k+1) $
	or 
	$ \widehat{\eta}(2k) \neq \widehat{\eta}(2k+1) $.
	If 
	$ \widehat{\theta}(2k) \neq \widehat{\theta} (2k+1) $
	then
	$ \varepsilon ( \widehat{\theta} (2k)) = \varepsilon (\widehat{\theta}(2k+1))$,
	which contradicts part $(ii)$ of Lemma \ref{lemma:prop:3:1}. 
	If 
	$ \widehat{\theta} (2k) = \widehat{\theta} (2k+1) $
	but
	$ \widehat{\eta}(2k) \neq \widehat{\eta}(2k+1)$,
	then 
	$ | \widehat{\eta}(2k+1) - \widehat{\eta}(2k) |  = 2 $.
	Since 
	$ \widehat{\eta} \circ p \circ \widehat{\eta} \vee 
	\widehat{\eta} \circ q\circ  \widehat{\eta} $ 
	is non-crossing, it follows that it has a block with exactly one element, which again contradicts Lemma \ref{lemma:prop:3:1}.
\end{proof}

\begin{lemma}\label{lemma:2:4}
	Suppose that 
	$ \theta(s)= 1 $ 
	and 
	$ \eta(s) = -1 $
	for all 
	$ s \in [ m ] $.
	If
	$p, q$
	are such that
	$ \mathcal{W}_{\theta, \eta} (p, q, I_M) \neq o(M^0)$, 
	then 
	$ p\vee q $
	is non-crossing and its blocks are of one of the following types:
	\begin{enumerate}
		\item[$(i)$] $(2k-1, 2k)$, 
		for some 
		$ k \in [ m ] $;
		\item[$(ii)$] $ (2k-1, 2k, 2t-1, 2t) $
        for some 
		$ k, t \in [m ]$
		with
		$ k < t $;
		\item[$(iii)$] pairs of the type 
		$ (2k-1, 2t)$, $(2k, 2t-1)$,
		for some 
		$ k, t \in [m ]$
		with
		$ k < t $.
	\end{enumerate}
	Moreover, we have that: 
	\begin{enumerate}
		\item[$\mybullet$] if 
		$ p \vee q $
		has a block of the form 
		$ (2k-1, 2k)$, then 
		$ \alpha_{2k-2} = \alpha_{2k+1} $
		and
		$ \beta_{2k-2} = \beta_{2k+1} $;
		\item[$\mybullet$] if
		$ p \vee q $
		has a block of the form 
		$ (2k-1, 2k, 2t-1, 2t) $
		with
		$ k< t $, 
		then
		$q(2k-1) = 2k $
		and
		$ \alpha_{2k-2} = \alpha_{2t+1} $,
		$ \beta_{2k-2} = \beta_{2t+1} $,
		$ \alpha_{2k+1}= \alpha_{2t-2}$,
		$ \beta_{2k+1} = \beta_{2t-2}$;
		\item[$\mybullet$] if
		$ p \vee q $
		has the blocks 
		$ (2k-1, 2t)$, $(2k, 2t-1)$,
		for some 
		$ k < t $
		then again
		$ \alpha_{2k-2} = \alpha_{2t+1} $,
		$ \beta_{2k-2} = \beta_{2t+1} $,
		$ \alpha_{2k+1}= \alpha_{2t-2}$,
		$ \beta_{2k+1} = \beta_{2t-2}$.
	\end{enumerate}
\end{lemma}
\begin{proof}
	The condition 
	$ \theta(s)= 1 $ 
	for all
	$ s $ 
	gives that 
	$ ( \widehat{\theta} \circ p \circ \widehat{\theta},
	\widehat{\theta} \circ q \circ \widehat{\theta})  = (p, q ) $
	so,  according to Corollary \ref{cor:2:3},
	$ p \vee q $
	is non-crossing. 
	So
	$ p \vee q $ 
	has some block with all elements consecutive, that we shall denote by
	$ B $.
	It suffices to show that 
	$ B $
	satisfies the properties above, and the conclusion follows by eliminating the block 
	$ B $ 
	and induction on 
	$ m $.
	
	If for all
	$k $ 
	we have that the pair 
	$ (2k, 2k+1)$ 
	is not in 
	$ B$,
	then Lemma \ref{lemma:prop:3:1} gives that 
	$ B $ is as in case $(i)$.  Suppose that 
	$ k \in [ m ] $ 
	is such that 
	$ B $ 
	contains the pair
	$ (2k, 2k +1) $.
	Note that 
	$ \widehat{\eta} (B) $
	is a block in 
	$ \widehat{\eta} \circ p \circ \widehat{\eta} \vee
	\widehat{\eta} \circ q \circ \widehat{\eta} $. 
	Since
	$ \widehat{\eta}(2k) = 2k-1 $
	and
	$ \widehat{\eta}(2k+1) = 2k+2$,
	it follows that the pair
	$( 2k-1, 2k+2)$
	is included in 
	$ \widehat{\eta}(B)$.
	But 
	$ \widehat{\eta} \circ p \circ \widehat{\eta} \vee
	\widehat{\eta} \circ q \circ \widehat{\eta} $
	is non-crossing and does not have singleton blocks (Lemma \ref{prop:3:1}), so either 
	$ (2k, 2k+1)$ 
	is a block of 
	$ \widehat{\eta} \circ p \circ \widehat{\eta} \vee
	\widehat{\eta} \circ q \circ \widehat{\eta} $,
	or the pair
	$ (2k, 2k+1) $ 
	is included in 
	$ \widehat{\eta}(B) $.
	
	If 
	$ (2k, 2k+1)$ 
	is a block of 
	$ \widehat{\eta} \circ p \circ \widehat{\eta} \vee
	\widehat{\eta} \circ q \circ \widehat{\eta} $,
	then 
	$ (\widehat{\eta}(2k), (\widehat{\eta}(2k+1)) 
	=(2k-1, 2k+2)$
	is a block of 
	$ p\vee q $, 
	so case $(iii)$ occurs.
	
	If the pair
	$ (2k, 2k+1) $ 
	is included in 
	$ \widehat{\eta}(B) $, 
	then 
	$ \widehat{\eta}(B) $
	contains the entire $4$-tuple 
	$ (2k-1, 2k, 2k+1, 2k+2)$
	and so does 
	$ B $.
	Applying Lemma \ref{prop:3:1} to 
	$ p \vee q $ gives that
	\begin{align}
	&\{ p(2k), q(2k) \} = \{ 2k-1, 2k+1\} \label{o1}\\
	&\{ p(2k+1), q(2k+1)\} =\{ 2k, 2k+2\} \label{o2}
	\end{align}
	and also that
	$ 2k \in \{ p(2k-1), q(2k-1)\}$
	and
	$ 2k+1 \in \{ p(2k+1), q(2k+1)\}$.
	On the other hand, applying Lemma \ref{prop:3:1}
	to 
	$ \widehat{\eta}\circ p \circ \widehat{\eta} 
	\vee
	\widehat{\eta}\circ q \circ \widehat{\eta}$
	we get that
	$ 2k \in \{ \widehat{\eta}\circ p \circ \widehat{\eta}(2k+1),
	\widehat{\eta}\circ q \circ \widehat{\eta}(2k+1) \} $,
	that is
	$ 2k-1\in \{ p(2k+2), q(2k+2)\} $.
	It follows that 
	\begin{align}
	& \{ p(2k-1), q(2k-1)\}= \{ 2k, 2k+2\}\label{o3}\\
	&\{ p(2k+2), q(2k+2)\} =\{ 2k-1, 2k+1\}\label{o4}.
	\end{align}
	Relations (\ref{o1})--(\ref{o4}) and Lemma \ref{prop:3:1} give that 
	$ B = (2k-1, 2k, 2k+1, 2k+2)$, i.e. case $(ii)$ occurs.
	
	Assume now that 
	$ (2k-1, 2k, 2t-1, 2t) $
	is a block of 
	$ p \vee q $ 
	and that 
	$ p(2k-1) = 2k $.
	Since 
	$ \theta(s)= 1 $ 
	for all
	$ s $, 
	with the notations from the proof of Lemma \ref{lemma:W:V},
	we obtain that 
	$ \alpha_{2k-1}= \alpha_{2k}$, 
	that 
	$\alpha_{2t-1}  = \alpha_{2t}$
	and that 
	$ \xi^\prime_{2k-1} = \xi^\prime_{2k}= \xi^\prime_{2t-1}= \xi^\prime_{2t} $
	so 
	$ \big|  \iota_B ( \mathcal{A}^{(p, q)}_{\textrm{id}, \theta, b})\big| \leq b^3 $ 
	and the conclusion follows from 
	Lemma \ref{lemma:segment}.
	
	For the last part of the Lemma, note that the conditions 
	$ \theta(s) = 1 $ 
	and 
	$ \eta(s) = -1 $
	give that 
	$ \alpha_{2s} = \alpha_{2s+1} $
	and
	$ \beta_{2s-1}=\beta_{2s+2} $
	for all 
	$ s \in [ m-1] $.
	If 
	$ ( 2k-1, 2k) $
	is a block of 
	$ p \vee q $
	(as in case $(i)$),
	then we also have that
	$ \alpha_{2k-1} = \alpha_{2k} $
	and
	$ \beta_{2k-1}= \beta_{2k} $,
	so the conclusion follows. The argument for case $(iii)$ is similar. For case $(ii)$, the condition 
	$ q(2k-1) = 2k $
	gives that 
	$ p(2k-1) = 2t $
	so 
	$ \alpha_{2k-1} = \alpha_{2t}$,
	hence 
	$ \alpha_{2k-2} = \alpha_{2t+1} $;
	the equality 
	$ \beta_{2k-2}= \beta_{2t+1} $
	follows similarly.
\end{proof}

\section{Main results}

\subsection{Framework}\label{framework}
Throughout this Section, we will suppose that
$ \big( A_{1, N}, \dots, A_{R, N}\big)_N $ 
is a sequence of $ R $-tuples of random matrices such that:
\begin{enumerate}
	\item[($\mathfrak{c}1$)] For each 
	$ N $ 
	we have that  
	$ A_{1, N}, A_{1, N}^{\ast},\dots, A_{R, N}, A_{R, N}^{\ast}  $ 
	is a \textit{unitarily invariant} 
	$ 2R $-tuple of
	$ M_N \times M_N $ 
	random matrices, with
	$ M_N = b_N \cdot d_N $
	and the sequences
	$\big(b_N )_N $, $ \big(d_N \big)_N $
	are increasing and unbounded.
	\item[($\mathfrak{c}2$)] The sequence 
	$ \big( A_{1, N}, \dots, A_{R, N}\big)_N $ 
	has the \emph{bounded cumulants property} from \cite{mingo-popa-transpose}, that is for any  integer 
	$ r \geq 2 $ 
	and any non-commutative polynomials with complex coefficients 
	$ p_0, p_1, \dots, p_r \in \mathbb{C}\langle X_1, Y_1,\dots, X_R, Y_R \rangle $ 
	we have that:
	\begin{enumerate}
		\item [$(i)$]$ \displaystyle
		\limsup_{N \rightarrow \infty } 
		|\E \circ \tr \big( 
		p_0(A_{1, N}, A_{1, N}^\ast, \dots, A_{R, N}, A_{R, N}^\ast )
		\big)| < \infty $
		\item[$(ii)$] 
		$ \displaystyle 
		\limsup_{N \rightarrow \infty } 
		| k_r \big( 
		\Tr (p_1(A_{1, N}, A_{1, N}^\ast, \dots, A_{R, N},A_{R,N}^\ast ), 
		\dots, \\
		{}\hspace{6cm}	 	
		\Tr( p_r (A_{1, N}, A_{1, N}^\ast, \dots, A_{R, N}, A_{R, N}^\ast )\big) 
		\big) | < \infty 
		$,\\
		where 
		$k_r $ 
		denotes the $r$-th (classical) cumulant. 	 	 
	\end{enumerate} 
\end{enumerate}

\begin{remark}\label{rem5:1}
Condition ($\mathfrak{c}2$)$(i)$ implies the following property. If 
$ p_1, \dots, p_m $ 
are some non-commutative polynomials in $2R $ variables with complex coefficients and, for each $k$ we denote
 $ T_{k, N}= p_k(A_{1, N}, A_{1, N}^\ast, \dots, A_{R, N}, A_{R, N}^\ast)$, 
 then
 \begin{align*}
\limsup_{N \rightarrow \infty} | \E \big( \tr(T_{1, N})\tr(T_{2, N}) \cdots \tr(T_{m, N}) \big) |
< \infty.  \qquad
\end{align*} 
\end{remark}
\begin{proof}
If 
$ m =1 $, 
the result is trivial. If 
$ m \geq 2 $
then
 the Cauchy-Schwarz inequality gives that 
$| \tr (T_{k, N}) |^2 \leq \tr (T_{k, N}^\ast T_{k, N}) $,
hence, denoting
$ \displaystyle  B_{ N}= \sum_{k=1}^m T_{k, N}^\ast T_{k, N} $, 
we have that
\begin{align*}
| \tr(T_{1, N}) \cdot \tr(T_{2, N}) \cdots \tr(T_{m, N}) | \leq [\tr (B_N)]^{m/2}.
\end{align*}
On the other hand, 
$B_N \geq 0 $ 
and the function 
$ t \mapsto t^{m/2} $
is convex on 
$ [ 0, \infty) $,
hence
\begin{align*}
[\tr (B_N)]^{m/2} \leq \tr ( B_N^{m/2}) < \tr (B_N^m + I)
\end{align*}
and the conclusion follows.
\end{proof}

Let
 $\vec{r} =(r_1, \dots, r_m) \in [ R]^m $ and
 $\vec{\nu}=(\nu_1, \dots, \nu_m) \in \{ 1, \ast\}^m$;
 we shall write
$ \mathcal{W}_{\theta, \eta}( p, q,  A_{\vec{r}}^{\vec{\nu}})$
for
$ \mathcal{W}_{\theta, \eta}( p, q,  A_{r_1}^{\nu_1},
A_{r_2}^{\nu_2}, \dots,\ab A_{r_m}^{\nu_m})$, and
respectively
 $ \Tr\big( A_{\vec{r}}\big)$
 for
 $ \Tr\big( A_{r_1}, \dots, A_{r_m} \big) $
 etc.
 
\begin{lemma}
	If, as in Section \ref{aux}, 
	$\theta, \eta: [ m ] \rightarrow\{-1, 1 \}$ and
	$ p, q $
	are pair-partitions in 
	$ \cP_2^{\varepsilon} (2m) $ 
	such that 
	$ \mathcal{W}_{\theta, \eta}(p, q, I_{M_N}) = o(M_N^0)$, 
	then for any sequence of $ R $-tuples of
	random matrices 
	$( A_{1, N}, \dots, A_{R, N}) $
satisfying conditions $(\mathfrak{c}1)$ and $(\mathfrak{c}2)$ $(i)$ we have that
	$ \mathcal{W}^{\nu}_{\theta, \eta}(p, q, A_{1, N}, \dots, A_{m, N}) = o(M_N^0)$.
\end{lemma}
\begin{proof}
	With the notations from Section \ref{aux}, we have that
	\begin{align*}
	\Tr_{\widehat{q}}(A_{1, N},\dots, A_{1, N})& = 
	M^{|\widehat{q}|} 
	\cdot
	\tr_{{\widehat{q}}}
	(A_{1, N},\dots, A_{m, N})\\
	&= \Tr_{\widehat{q}} (I_{M_N}, \dots, I_{M_N}) 
	\cdot
	\tr_{{\widehat{q}}}
	(A_{1, N},\dots, A_{m, N}),
	\end{align*}
that is
	\begin{align*}
	\mathcal{W}_{\theta, \eta}  (p, q, A_{1, N}, \dots, A_{m, N}) = 
	\mathcal{W}_{\theta, \eta} (p, q, I_{M_N})	
	\cdot
	\big[ \E \circ \tr_{\widehat{q}} (A_{1, N}, \dots, A_{m, N})\big].
	\end{align*} 
	and the conclusion follows from Remark \ref{rem5:1}
\end{proof}

 \subsection{Limit distribution of the partial transpose.}\label{section:limitdistrib}
 
 \begin{theorem}\label{thm:limdistrib}
 Suppose that $ \big( A_{1, N}, \dots, A_{R, N}\big)_N $ 
 is a sequence of $ R $-tuples of random matrices satisfying conditions $( \mathfrak{c}1)$ and $(\mathfrak{c}2)$ and that
 for each
 $ \nu_0, \nu_1 , \nu_2 \in \{ 1, \ast \}$
 and
 $ k, l \in [ R ] $
 there exists some real numbers
 $ m_{(\nu_0)}(k)$
 and 
 $ m_{(\nu_1, \nu_2)}(k, l) $
satisfying
 \begin{align*}
 &\lim_{N\rightarrow \infty} \E \circ \tr \big( A_{k, N}^{\nu_0} \big) = m_{(\nu_0)}(k)\\
 &\lim_{N\rightarrow \infty} \E \circ \tr \big( A_{k, N}^{\nu_1}A_{l, N}^{\nu_2} \big) = m_{(\nu_1, \nu_2)} (k, l).
 \end{align*}
 Then, for any
 	$ \vec{\nu} = 
 	( \nu_1, \dots, \nu_m) \in \{ 1, \ast\}^m $
 	and any
 	$\vec{r} =(r_1, \dots, r_m) \in [ R]^m $,
 	we have that
 	\begin{multline}\label{fc:gamma}
 	\lim_{N \rightarrow \infty } \kappa_m \big(
 	( A_{r_1, N}^{ \rtr})^{\nu_1} ,
 	\dots,
 	( A_{r_m, N}^{ \rtr})^{\nu_m} 
 	\big)  \\	
 	= 
 	\begin{cases}
 	m_{(\nu_1)}(r_1), & \textrm{ if } m =1 \\
 	m_{( \nu_1, \nu_2 )}(r_1, r_2) - m_{(\nu_1)}(r_1) m_{(\nu_2)}(r_2),
 	& \textrm{ if } m =2\\
 	0,  & \textrm{ if } m \geq  3
 	\end{cases}
 	\end{multline}
 	where 
 	$ \kappa_m\big(
 	X_1 ,
 	X_2 ,
 	\dots,
 	X_m
 	\big) $ 
 	denotes the $m^{th}$ free cumulant of 
 	the tuple
 	$\big(
 	X_1, X_2, \dots, X_m
 	\big) $ with respect to $\E \circ \tr$.
 \end{theorem}

We shall prove Theorem \ref{thm:limdistrib} in several steps. First, we shall introduce several notations for some classes of non-crossing partitions. Then we shall derive 
(\ref{fc:gamma}) using Lemma \ref{lemma:2:4} and the moment-free cumulant recursion.

\begin{proof}
As before, we shall denote by 
$ NC(m) $
the set of non-crossing partitions on the ordered set 
$ [m]$. 
Furthermore, denote
\begin{align*}
& NC_{1,2}(m)=
\{ \pi \in NC(m) \mid \textrm{each block of $\pi $ has either one or 2 elements} \},
\\
&AP(2m) =
\{ \pi \in NC(2m) \mid \textrm{ the blocks of }
\pi  \textrm{ satisfy conditions $(i)$--$(iii)$ of Lemma \ref{lemma:2:4}} \}.
\end{align*}

Next, consider the surjective mapping 
$ \mathfrak{i}:AP(2m) \rightarrow NC_{1, 2}(m) $
given as follows. 
For 
$ \pi \in AP(2m) $,
we have that 
$ (k) $ 
is a block of 
$ \mathfrak{i}(\pi) $
whenever 
$ (2k-1, 2k)$ 
is a block of 
$ \pi $,
and that
$ (k, t)$
is a block of 
$ \mathfrak{i}(\pi) $
whenever either 
$ (2k-1, 2k, 2t-1, 2t )$ 
is a block of 
$ \pi $,
or $(2k-1, 2t )$ and $ (2k, 2t-1)$ 
are both blocks of 
$ \pi $.	

From the moment-free cumulant recursion formula (see \cite{nica-speicher}), it suffices to show that 
\begin{multline}\label{A:m}
\lim_{N \rightarrow \infty} \E\big(\tr  \big( 
( A_{r_1, N}^{ \rtr})^{\nu_1} \cdot 
( A_{r_2, N}^{ \rtr})^{\nu_2} 
\cdots
( A_{r_m, N}^{ \rtr})^{\nu_m}
\big)\\
=  \sum_{\pi\in NC_{1, 2}(m)}
\prod_{\substack{B \in \pi\\ B = (k)} } 
m_{(\nu_k)}(r_k)
\prod_{\substack{B \in \pi\\ B = (s, t)} } 
\big(m_{(\nu_s, \nu_t)}(r_s, r_t) - m_{(\nu_1)}(r_1) m_{(\nu_2)}(r_2)\big).
\end{multline}

Corollary \ref{cor:2:3} and Lemma \ref{lemma:2:4}, give that 
\begin{align*}
\lim_{N \rightarrow \infty} \E \circ  \tr  \big( 
( A_{r_1, N}^{ \rtr})^{\nu_1} \cdot
( A_{r_2, N}^{ \rtr})^{\nu_2} 
\cdots
( A_{r_m, N}^{ \rtr})^{\nu_m}
\big)
=
\lim_{N \rightarrow \infty}
\sum_{ \substack{p, q \in \cP_2^{\varepsilon}(2m) \\ 
		p \vee q \in AP(2m)} }
\mathcal{W}_{\theta, \eta}(p, q, A_{\vec{r}}^{\vec{\nu}}),
\end{align*}
where
$ \theta$ and $ \eta $
are given by
$ \theta(s) = 1 $ 
and
$ \eta(s) = -1 $
for all 
$ s\in[m] $.
Therefore, to show (\ref{A:m}), it suffices to show that, 
for each
$ \rho \in NC_{1, 2}(m) $, the following relation holds true:
\begin{align}\label{W:m}
\sum_{ \substack{p, q \in \cP_2^{\varepsilon}(2m) \\ 
		\mathfrak{i} (p \vee q )
		= \rho} }
\lim_{N \rightarrow \infty}
\mathcal{W}_{\theta, \eta}( & p, q,  A_{\vec{r}}^{\vec{\nu}})\\
= & \prod_{\substack{B \in \rho\\ B = (k)} } 
m_{(\nu_k)}(r_k)
\prod_{\substack{B \in \rho\\ B = (s, t)} } 
\big( m_{(\nu_s, \nu_t)}( r_s, r_t) - 
m_{(\nu_s)}( r_s) \cdot m_{(\nu_t)}( r_t)
\big).
\nonumber
\end{align}

We shall prove (\ref{W:m}) by induction on 
$ m $. 
Since 
$ \rho $ is non-crossing, it has at least one block which is a segment. Denote by 
$ B $ 
the block of 
$ \rho $ 
which is a segment and has the property that the least element of 
$B$ is minimal among all segments.
Suppose first that 
$ B$ 
is a singleton, i.e. 
$ B =  (k) $ 
for some 
$ k \in [ m ] $.
From the construction of the mapping 
$ \mathfrak{i} $, 
it follows that 
$ p $ and $ q $ 
are in the case $(i)$ of Lemma \ref{lemma:2:4},
so 
$ p(2k-1) = q(2k-1) = 2k $;
in particular, the set 
$ [ 2m] \setminus \{ 2k-1, 2k\} $
is invariant by
$ p, q  $. Also, remember that the map 
$ \widehat{q} \in S(m) $
from Equation (\ref{it:W}) is defined via 
$ \widehat{q}(s) = \frac{1}{2}(q(2s) +1) $,
so 
$ (k) $ 
is also a singleton cycle of $ \widehat{q} $. 
Hence, denoting by 
$ p_0, q_0 $
the restrictions of 
$ p, q $ 
to the set 
$ [ 2m] \setminus \{ 2k-1, 2k\} $
and by 
$ \vec{r}^{\mb\prime} $,
 respectively 
 $\vec{\nu}^{\mb\prime} $
the multi-indeces
$(r_1, \dots, r_{k-1}, r_{k+1}, \dots, r_m)$
and 
$(\nu_1, \dots, \nu_{k-1}, \nu_{k+1}, \dots, \nu_m )$,
we get that
\begin{align*}
\tr_{\widehat{q} } 
(A_{\vec{r}}^{\vec{\nu}})
= \tr(A_{r_k, N}^{\nu_k}) \cdot 
\tr_{\widehat{q_0} }
(A_{\vec{r}^{\mb\prime}}^{\vec{\nu}^{\mb\prime}}),
\end{align*}
which, from conditions $(\mathfrak{c}1)$ and $(\mathfrak{c}2)$, give that
\begin{align}\label{tr:sep1}
\lim_{N \rightarrow \infty}\E \big( 
\tr_{\widehat{q} } (A_{\vec{r}}^{\vec{\nu}} )
\big) 
=  m_{( \nu_k)} (r_k)
\cdot
\lim_{N \rightarrow \infty}
\E \big(
\tr_{\widehat{q_0} }
(A_{\vec{r}^{\mb\prime}}^{\vec{\nu}^{\mb\prime}} )
\big). 
\end{align}

Denote by
$ h $ 
the pairing on 
$ \{1, 2\}$ (that is $ h(1) = 2 $)
and  note that, since according to Lemma \ref{lemma:2:4}, 
$ \alpha_{2k-2} = \alpha_{2k+1} $
and
$ \beta_{2k-2} = \beta_{2k+1} $,
we have that 
\begin{align*}
\lim_{N \rightarrow \infty}
\mathcal{W}_{\theta, \eta} (p, q, I_M) & = 
\lim_{N \rightarrow \infty }
\big[
\mathcal{W}_{\theta_0, \eta_0} (p_0, q_0, I_M)
\cdot 
\textrm{Wg}_M (h, h) \cdot 
\sum_{ \substack{\xi_{2k-1}^\prime, \xi_{2k}^\prime \in [ b ] \\
		\xi_{2k-1}^{\prime\prime}, \xi_{2k}^{\prime\prime} \in [ d ]		
} }
\delta_{\xi_{2k-1}^\prime}^{\xi_{2k}^\prime} 
\delta_{\xi_{2k-1}^{\prime\prime}}^{\xi_{2k}^{\prime\prime} }
\big]\\
& = 
\lim_{N \rightarrow \infty }
\big[
\mathcal{W}_{\theta_0, \eta_0} (p_0, q_0, I_M)
\cdot 
M \cdot \textrm{Wg}_M (h, h)
\big]\\
& = 
\lim_{N \rightarrow \infty }
\mathcal{W}_{\theta_0, \eta_0} (p_0, q_0, I_M)
\end{align*}
hence, according to relation (\ref{tr:sep1}), equality (\ref{W:m}) follows by induction on 
$ m $.

Suppose now that 
$ B = (k, k+1) $
for some 
$ k \in [ m-1] $.
Denote
\begin{align*}
\cP_1(\rho) &
= \big\{ \pi \in \iota^{-1} ( \{ \rho\} ) : \  
(2k-1, 2k, 2k+1, 2k+2) \textrm{ is a block of } \pi
\big\} \\
\cP_2(\rho) &
= \big\{ 
\pi \in \iota^{-1} ( \{ \rho\} ) : \
(2k-1, 2k+2)  \textrm{ and } (2k, 2k+1) \textrm{ are blocks of } \pi
\big\}.
\end{align*}

As in the previous case, let 
$ p, q \in \cP_2^{\varepsilon}(2m) $
be such that 
$ p \vee q = \pi $.
We shall distinguish two cases.

\noindent
\fbox{1}
If 
$ \pi \in \cP_1 (\rho) $, 
then 
$ p $ and  $ q $ 
are as in case $(ii)$ of Lemma \ref{lemma:2:4}, 
so we have that
$ p(2k-1 ) =2k+2  $, 
$ p( 2k) = 2k+1$,
$ q(2k-1) = 2k $, 
$ q( 2k+1) = 2k + 2 $
and the set 
$ [ 2m] \setminus \{ 2k-1, 2k, 2k+1, 2k+2 \} $
is invariant under 
$ p $ 
and 
$ q $.
Next, remark that, by construction
$ (k) $
and
$( k +1) $
are blocks of 
$ \widehat{q} $ 
and that denoting, this time, by 
$p_0, q_0$
the restrictions of 
$ p, q $
to the set 
$ [ 2m ] \setminus \{ 2k-1, 2k, 2k+1, 2k+2 \}$
and by 
$ \vec{r}^{\mb\prime} $,
respectively
$ \vec{\nu}^{\mb\prime} $
the multi-indeces
$(r_1, \dots, r_{k-1}, r_{k+2}, \dots, r_m)$
and
$(\nu_1, \dots, \nu_{k-1}, \nu_{k+2}, \dots, \nu_m)$
we get that 
 \begin{align*}
\tr_{\widehat{q} } (A_{\vec{r}}^{\vec{\nu}} )
= \tr(A_{r_k, N}^{\nu_k}) \tr ( A_{r_{k+1}, N}^{\nu_{k+1}} ) \cdot 
\tr_{\widehat{q_0} }
(A_{\vec{r}^{\mb\prime}}^{\vec{\nu}^{\mb\prime}}),
\end{align*}
which, using conditions $(\mathfrak{c}1)$ and $(\mathfrak{c}2)$ gives
\begin{align}\label{tr:sep2}
\lim_{N \rightarrow \infty}\E \big( 
\tr_{\widehat{q} }  (A_{\vec{r}}^{\vec{\nu}})
\big) 
= m_{ ( \nu(k) )}(r_k) m_{ ( \nu(k+1))} (r_{k+1})
\cdot
\lim_{N \rightarrow \infty}
\E \big(
\tr_{\widehat{q_0} }(
A_{\vec{r}^{\mb\prime}}^{\vec{\nu}^{\mb\prime}})
\big). 
\end{align}	
On the other hand, denoting by 
$ h_1 $,
respectively
$ h_1^\prime $
the pairings on 
$ \{ 1, 2, 3, 4 \} $ 
given by 
$ h_1 (1) = 4 $, $ h_1 (2) = 3 $,
respectively 
$ h_1^\prime (1) = 2 $, $ h_1^\prime (3) = 4 $,
and using again that, according to Lemma \ref{lemma:2:4},
$\alpha_{2k-2} = \alpha_{2k+3}$ 
and 
$ \beta_{2k-2} = \beta_{2k+3}$,
we have that
 \begin{align*}
\lim_{N \rightarrow \infty}
\mathcal{W}_{\theta, \eta} (p, q, I_M) & = 
\lim_{N \rightarrow \infty }
\big[
\mathcal{W}_{\theta_0, \eta_0} (p_0, q_0, I_M)
\cdot M^3 \cdot 
\textrm{Wg}_M (h_1, h_1^\prime) \\
& = - 
\lim_{N \rightarrow \infty }
\mathcal{W}_{\theta_0, \eta_0} (p_0, q_0, I_M),
\end{align*}
 since Equation (\ref{w-mult}) gives that 
$ \displaystyle 
\lim_{N \rightarrow \infty} M^3
\cdot 
\textrm{Wg}_M (h_1, h_1^\prime)
= -1 $.
Hence, using (\ref{tr:sep2}) we obtain that
  \begin{align}\label{cum2:1}
\lim_{N \rightarrow \infty }\mathcal{W}_{\theta, \eta} (p, q, A_{\vec{r}}^{\vec{\nu}}) = 
- m_{(\nu(k) )}(r_k) m_{\nu(k+1)} (r_{k+1})
\cdot 
\lim_{N \rightarrow \infty }
\mathcal{W}_{\theta_0, \eta_0} (p_0, q_0, A_{\vec{r}^{\mb\prime}}^{\vec{\nu}^{\mb\prime}}).
\end{align}

\noindent
\fbox{2}
If 
$ \pi \in \cP_2 (\rho) $, 
then 
$ p $ and  $ q $ 
are as in case $(iii)$ of Lemma \ref{lemma:2:4}, so
$ p(2k-1) = q (2k-1) = 2k+2 $ 
and
$ p(2k) = q(2k) = 2k+1 $.
In this case, 
$ (k, k+1) $
is a block of 
$ \widehat{q} $ 
and (again
$ p_0, q_0 $ 
denote the restriction of
$ p, q $ 
to the set 
$ [ 2m] \setminus \{ 2k-1, 2k, 2k+1, 2k+2 \} $):
\begin{align}\label{tr:sep3}
\lim_{N \rightarrow \infty}\E \big( 
\tr_{\widehat{q} } (A_{\vec{r}}^{\vec{\nu}})
\big)
= m_{ ( \nu(k),\nu(k+1))} (r_k, r_{k+1})
\cdot
\lim_{N \rightarrow \infty}
\E \big(
\tr_{\widehat{q_0} }
(A_{\vec{r}{\mb\prime}}^{\vec{\nu}^{\mb\prime}})
\big). 
\end{align}	

Denoting by 
$ h_2 $
the pairing on 
$ \{ 1, 2, 3, 4\}$
given by
$ h_2 (1) = 4 $ 
and 
$ h_2(2) =3 $,
then using Lemma \ref{lemma:2:4} and Equation (\ref{w-mult}) as above, we get
\begin{align*}
\lim_{N \rightarrow \infty}
\mathcal{W}_{\theta, \eta} (p, q, I_M) & = 
\lim_{N \rightarrow \infty }
\big[
\mathcal{W}_{\theta_0, \eta_0} (p_0, q_0, I_M)
\cdot M^2 \cdot 
\textrm{Wg}_M (h_2, h_2) \big] \\
& = 
\lim_{N \rightarrow \infty }
\mathcal{W}_{\theta_0, \eta_0} (p_0, q_0, I_M),
\end{align*}
therefore
\begin{align}\label{cum2:2}
\mathcal{W}_{\theta, \eta} (p, q, A_{\vec{r}}^{\vec{\nu}}) = 
m_{(\nu(k), \nu(k+1))} (r_k, r_{k+1})
\cdot 
\lim_{N \rightarrow \infty }
\mathcal{W}_{\theta_0, \eta_0} (p_0, q_0, A_{\vec{r}^{\mb\prime}}^{\vec{\nu}{\mb\prime}}).
\end{align}

Denoting by 
$\rho_0 $ 
the restriction of 
$ \rho $
to 
$ [2m] \setminus \{ 2k-1, 2k, 2k+1, 2k+2 \}$,
Equations (\ref{cum2:1}) and (\ref{cum2:2}) give
 \begin{align}
\sum_{ \substack{p, q \in \cP_2^{\varepsilon}(2m) \\ 
		\mathfrak{i} (p \vee q )
		= \rho} }
\lim_{N \rightarrow \infty}
\mathcal{W}_{\theta, \eta}(p, q, A_{\vec{r}}^{\vec{\nu}})
= 
\big[
m_{(\nu(k), \nu(k+1) )}&(r_k, r_{k+1}) -  m_{(\nu(k) )}(r_k) m_{\nu(k+1)}(r_{k+1})
 \big] \cdot \\
&
\sum_{ \substack{p_0, q_0 \in \cP_2^{\varepsilon}(2m-4) \\ 
		\mathfrak{i} (p_0 \vee q_0 )
		= \rho_0} }
\lim_{N \rightarrow \infty}
\mathcal{W}_{\theta, \eta}(p_0, q_0, A_{\vec{r}^{\mb\prime}}^{\vec{\nu}^{\mb\prime}})\nonumber
\end{align}
and the conclusion follows again by induction on 
$ m $.
\end{proof}

\begin{remark}
	Since Wishart ensembles of random matrices have the bounded cumulant property (see \cite{wishart2}), the results from \cite{aubrun} are a particular case of Theorem \ref{thm:limdistrib}.
\end{remark}

\begin{remark}
	If 
	$  \big(A_N \big)_N $
	is an ensemble of self-adjoint random matrices that satisfies the hypothesis of Theorem \ref{thm:limdistrib}, then the limit distribution of  
	$ A_N^\rtr $
	is a translated semicircular distribution.
	If the ensemble
	$ \big( A_N \big)_N $ 
	is asymptotically $R$    -diagonal, then
	$ A_N^\rtr $ 
	is asymptotically circular.
\end{remark}

\begin{remark}
	Theorem \ref{thm:limdistrib} allows us an easy formulation of necessary and sufficient conditions for 
	$ A_N^{\rtr} $ 
	to have a positive limit distribution (as analyzed in \cite{aubrun} for Wishart ensembles). More precisely, we need 
	\[ \left\{ 	\begin{array}{l}
	m_{(1)}= m_{(\ast)},  \ \ m_{(1, 1)} = m_{(1, \ast)} =m_{(\ast, \ast)} \\
	\\
	m_1^2 \leq 4 ( m_{(1, 1)} - m_{(1)}^2 ).
	\end{array}
	\right.\]
\end{remark}

\subsection{Asymptotic relations of free independence}

For the proof of the main result of this subsection, Theorem \ref{thm:free} below, we need to introduce a new notation in the spirit of Lectures 9 and 18 from \cite{nica-speicher}.

Let
$ S = \{ a(1), a(2), \dots, a(r) \}$
be a subset of 
$ [ m] $,
with
$ a(1) < \dots < a(r) $. 
If 
$ \rho $ 
is a non-crossing partition on the set 
$ S $,
we shall denote by
$ C_m ( \rho )  $
the largest (in the partial order given by inclusion of blocks) non-crossing partition on
$ [ m]\setminus S $
such that the partition consisting in the union of blocks of 
$ \rho $
and 
$ C_m(\rho)$
is non-crossing on 
$ [ m ]$.
If 
$ X_1, \dots, X_m $
are elements in some non-commutative probability space 
$ (\mathfrak{A}, \Phi)$,
then we denote
 \begin{align}\label{k:C}
\kappa_\rho \Phi_{C_m(\rho)}
\big[ X_1, & X_2, \dots, X_m \big]\\
= &
\prod_{ \substack{B \in \rho \\ B = \{ b(1), \dots, b(t)\}}} 
\kappa_t (X_{b(1)}, \dots, X_{b(t)})
\cdot
\prod_{ \substack{D \in C_m (\rho) \\ D = \{ d(1), \dots, d(q)\}}} 
\Phi (X_{d(1)} \cdots X_{d(q)} ).\nonumber
\end{align}

\begin{theorem}\label{thm:free}
 Suppose that $ \big( A_{1, \sN}, \dots, A_{\sR, \sN}\big)_N $ 
is a sequence of $ R $-tuples of random matrices satisfying conditions $( \mathfrak{c}1)$ and $(\mathfrak{c}2)$ 
and with converging joint $\ast$-moments, i.e.
$\displaystyle
\lim_{N \rightarrow \infty} E \circ \tr \big( 
A_{r_1, \sN}^{\nu_1} A_{r_2, \sN}^{\nu_2} \cdots A_{r_m, \sN}^{\nu_m}
\big)
$
exits and it is finite for any positive integer $m $,
any $\nu_1, \dots, \nu_m \in\{ 1, \ast\}$ 
and any $r_1, \dots, r_m \in [ R ] $.

Then the four ensembles 
$(A_{1, \sN}, \dots, A_{\sR, \sN})_N $, 
$ (A_{1, \sN}^\top, \dots, A_{\sR, \sN}^\top)_N $,
$ (A_{1, \sN}^\rtr, \dots, A_{\sR, \sN}^{\rtr})_N $
and 
$(A_{1, \sN}^{\ltr}, \dots, A_{\sR, \sN}^{\ltr})_N $
are asymptotically $\ast$-free from each other.
\end{theorem}

\begin{proof}
	The asymptotic $\ast$-free independence of 
$(A_{1, \sN}, \dots, A_{\sR, \sN})_N $ 
and
$(A_{1, \sN}^\top, \dots, A_{\sR, \sN}^\top)_N $ 
 was shown in \cite{mingo-popa-transpose}.  Since
$ A^{\ltr} =(A^\top)^\rtr$,
 it suffices to show that
 $ (A_{1, \sN}^\rtr, \dots, A_{\sR, \sN}^\rtr)_N $
is asymptotically free from the ensemble
$(A_{1,\sN}, A_{1, \sN}^\top, A_{1, \sN}^{\ltr}, \dots, A_{\sR,\sN}, A_{\sR, \sN}^\top, A_{\sR, \sN}^{\ltr})_N $.

We shall prove this last statement using the free cumulant-moment recursion formula, as detailed below.
  Let
$ m $ 
be a positive integer,  let
$ r:[m] \rightarrow [R]$,
$ \nu : [ m] \rightarrow \{ 1, \ast\} $
and let 
$ \theta, \eta : [ m ] \rightarrow \{ -1, 1\} $.

We shall use the notation (\ref{k:C}) for 
$\Phi= \E \circ \tr $
and
$  S = \{ a(1) , a(2), \dots, a(r) \}$
with
$ a(1) < a(2) < \dots a(r) $
be the set of all
$ s \in [ m ] $ 
such that 
$ \theta(s) =1 $
and 
$ \eta(s) = -1 $.
Since, according to Theorem \ref{thm:limdistrib}, all free cumulants of order higher than two in 
$ A_{k, \sN}^\rtr $ 
and 
$ (A_{k, \sN}^\rtr)^\ast $
vanish as 
$ N \rightarrow \infty $,
it suffices to show that
\begin{align}\label{k-E-tr}
\lim_{N \rightarrow \infty} 
\E \circ & \tr 
\big( 
\prod_{ s = 1 }^m ( A_{r_s, \sN}^{( \theta(s), \eta(s))}  )^{\nu(s)} 
\big)\\
& = \lim_{N \rightarrow \infty}
\sum_{ \rho \in NC_{1, 2}(S) } \kappa_{\rho} \cdot
(\E \circ \tr)_{ C_m(\rho)} 
\big[
 (A_{r_1, \sN}^{( \theta(1), \eta(1))})^{\nu(s)}
,
\dots,
 (A_{r_m, \sN}^{( \theta(m), \eta(m))})^{\nu(m)}
\big].\nonumber
\end{align}

On the other hand, with the notations from Section \ref{aux}, we have that
  \begin{align*}
\E \circ \tr
\big( 
\prod_{ s = 1 }^m ( A_{r_s, \sN}^{( \theta(s), \eta(s))}  )^{\nu(s)} 
\big)
=
\sum_{p, q \in \cP_2^{\varepsilon}(2m) } 
\mathcal{W}_{\theta, \eta} (p, q, A_{\vec{r}}^{\vec{\nu}}).
\end{align*}

According to Corollary \ref{cor:2:3}, 
$ \displaystyle\lim_{N \rightarrow \infty} \mathcal{W}_{\theta, \eta} (p, q, A_{\vec{r}}^{\vec{\nu}}) = 0$ 
unless
$\theta $ and $ \eta $ 
are constant on the blocks of 
$ p \vee q $, 
i.e. unless the set 
$\widetilde{S} = \{ 2s-1, 2s:\ s \in S \} $ 
is invariant under 
$ p $ and $ q $, 
i.e.unless the blocks of 
$ p_{| \widetilde{S}} \vee q_{| \widetilde{S}}$
satisfy the conditions from Lemma \ref{lemma:2:4}.
Therefore, with the notations from the proof of Theorem \ref{thm:limdistrib}, we have that
\begin{align}\label{concl2:0}
\lim_{N \rightarrow \infty}
\E \circ \tr \big( 
\prod_{ s = 1 }^m ( A_{r_s, \sN}^{( \theta(s), \eta(s))})^{\nu(s)}
\big) 
= 
\lim_{N \rightarrow \infty}
\sum_{ \rho \in NC_{1, 2}(S)}
\sum_{\substack{p, q \in \cP_2^{\varepsilon}(2m)\\
		p_{|\widetilde{S}} \vee q_{ | \widetilde{S}} = \rho}}
\mathcal{W}_{\theta, \eta} (p, q, A_{\vec{r}}^{\vec{\nu}}).
\end{align}

 Denote the first block of 
$ \rho $
which is a segment by
$ B(\rho)$.
Again as in the proof of Theorem \ref{thm:limdistrib}, we shall distinguish two cases: if 
$ B(\rho)$
is a singleton or if it is a pair.

If 
$ B(\rho) $
is a singleton, i.e. 
$ B(\rho) = (k) $ 
for some
$ k = a(s) $ 
with 
$ s \in [r] $,
then as shown in the proof of Theorem \ref{thm:limdistrib}, we have that
\begin{align*}
\lim_{N \rightarrow \infty}
\mathcal{W}_{\theta, \eta} (p, q, A_{\vec{r}}^{\vec{\nu}})
= 
m_{(\nu(k))} (r_k)\cdot
\lim_{N \rightarrow \infty}
\mathcal{W}_{\theta_0, \eta_0} (p_0, q_0, A_{\vec{r}^{\mb\prime}}) 
\end{align*}
where, as before
$ \theta_0, \eta_0 $
are the restrictions of 
$ \theta, \eta$
to
$ [m ] \setminus \{k\} $,
$\vec{r}^{\mb\prime}$, 
respectively
$\vec{\nu}^{\mb\prime}$
are the multi-indices
$(r_{1}, \dots, r_{k-1}, r_{k+1}, \dots, r_{m})$,
respectively
$(\nu(1), \dots, \nu(k-1), \nu(k+1), \dots,\ab \nu(m))$,
and
$ p_0, q_0 $
are the restrictions of
$ p, q $
to 
$ [2m] \setminus \{ 2k-1, 2k \}$.
Since there are no constraints on 
$ \rho_{ | S \setminus \{ k \} } $
other than being an element of 
$ NC_{1, 2}( S \setminus \{ k \} )$,
we obtain that 
\begin{align}\label{concl2:1}
\lim_{N \rightarrow \infty}
\sum_{\substack{ \rho \in NC_{1, 2,}(S) \\ B(\rho) = (k) }}
\sum_{\substack{p, q \in \cP_2^{\varepsilon}(2m)\\
		p_{|\widetilde{S}} \vee q_{ | \widetilde{S}} = \rho}}
\mathcal{W}_{\theta, \eta} &(p, q, A_{\vec{r}}^{\vec{\nu}})\\
=
& m_{(\nu(k))}(r_k) \cdot
\lim_{N \rightarrow \infty}
\E \circ \tr 
\big( 
\prod_{ \substack{1 \leq s \leq m \\
		s \neq k  } } 
(A^{( \theta(s), \eta(s))})^{\nu(s)}
\big).\nonumber
\end{align}

If 
$ B( \rho)$
is a pair, that is 
$ B(\rho) = (k, t) $
with 
$ k = a(s) $
and
$ t = a(s+1) $
for some 
$ s \in [ r-1] $,
then 
$p $ and $ q $ are as in the cases $(ii)$ and $(iii)$ of Lemma \ref{lemma:2:4}.
 In particular,
$ \alpha_{2k-2} = \alpha_{2t+1} $,
$ \beta_{2k-2} = \beta_{2t+1} $
and
$ \alpha_{2k+1}= \alpha_{2t-2}$,
$ \beta_{2k+1} = \beta_{2t-2}$,
hence the argument from the proof of Theorem \ref{thm:limdistrib} gives that 

\begin{align}\label{concl2:2}
\lim_{N \rightarrow \infty} 
\sum_{\substack{ \rho \in NC_{1, 2,}(S) \\ B(\rho) = (k) }} &
\sum_{\substack{p, q \in \cP_2^{\varepsilon}(2m)\\
		p_{|\widetilde{S}} \vee q_{ | \widetilde{S}} = \rho}} 
\mathcal{W}_{\theta, \eta} (p, q, A_{\vec{r}}^{\vec{\nu}})
  =  
m_{(\nu(k), \nu(t))} (r_k, r_t)\cdot \\
& \cdot \lim_{N \rightarrow \infty}
\E \circ \tr 
\big( 
\prod_{ k < s < t } 
 (A^{( \theta(s), \eta(s)})^{\nu(s)}
\big)
\cdot
\E \circ \tr 
\big( 
\prod_{ \substack{1< s < m\\
		s \notin [k, t]  } } 
 (A^{( \theta(s), \eta(s))})^{\nu(s)}
\big).\nonumber
\end{align}

Finally, Equation (\ref{k-E-tr}), and thus the conclusion, follows from
(\ref{concl2:0}), (\ref{concl2:1}), (\ref{concl2:2}) by induction on $m $.
\end{proof}
\begin{corollary}
Suppose that $(A_N)_N $ and $(B_N)_N $ are two sequences of random matrices such that each satisfies condition $(\mathfrak{c}1)$ and $(\mathfrak{c}2)$ and has converging $\ast$-moments. If the entries of  $ (A_{\sN} )_N $ and $ (B_{\sN})_N $ are independent, then the family 
$\{ A_{\sN}, B_{\sN}, A_{\sN}^\rtr, B_{\sN}^\rtr\}_N $ is asymptotically free.

In particular, if $ (A_{\sN})_N $ and $(B_{\sN})_N $ are two independent Wishart ensembles, then 
$ (A_{\sN}^\rtr)_{\sN} $ and $(B_{\sN}^\rtr)_{\sN}$ are asymptotically free and semicircularly distributed.
\end{corollary}

\begin{proof}
 Since $(A_{\sN})_N $ and $( B_{\sN})_N $ have independent entries, $(A_{\sN}, B_{\sN})_N $ is a sequence of unitarily invariant pairs. We wish to apply Theorem \ref{thm:free} which requires that $(A_{\sN}, B_{\sN})_N $ has the bounded cumulants property of $(\mathfrak{c}2)$\,$(ii)$, but we only have this separately for $(A_{\sN})$ and $(B_\sN)$. The same problem arose in the proof of Theorem 3.15 of \cite{mss} and we shall use the same reasoning here, with the vanishing cumulants property of \cite{mss} replaced by the bounded cumulants property of $(\mathfrak{c}2)$\,$(ii)$.  First we replace $B_\sN$ by $U_\sN B_\sN U^*_\sN$, which we can do because $B_\sN$ is assumed to be unitarily invariant. When we consider classical cumulants of traces of words in $A_\sN$ and  $U_\sN B_\sN U^*_\sN$, the Weingarten calculus then restricts consideration to words that consist entirely of $A_\sN$'s or entirely of $B_\sN$'s. Third, we can now apply the bounded cumulants property separately to traces of words in $A_\sN$ and traces of words in $B_\sN$'s to conclude that $A_\sN$ and $B_\sN$ have jointly the bounded cumulants property. 
 
 Now Theorem \ref{thm:free} gives that 
 $ ( A_{\sN}, B_{\sN})_N $ and $(A_{\sN}^\rtr, B_\sN^\rtr)_N $ 
 are asymptotically free from each other, hence it suffices to show that 
 $(A_{\sN}^\rtr)_N $ 
 and 
 $(B_{\sN}^\rtr)_N $ 
 are asymptotically free, i.e. that all their mixed free cumulants vanish. But, according to Theorem \ref{thm:limdistrib}, all the free cumulants of order higher than two vanish, while
 \begin{align*}
 \lim_{N \rightarrow \infty}
 \kappa_2(A_{\sN}^\rtr, B_{\sN}^\rtr) 
 = \lim_{N \rightarrow \infty} 
 \kappa_2(A_{\sN}, B_{\sN}) 
 \end{align*}
 and the conclusion follows.
\end{proof}







\bibliographystyle{alpha}

\end{document}